\documentclass[reqno,11pt]{amsart}
\usepackage{amsmath,amsthm}
\usepackage{amssymb}
\usepackage{xypic}
\usepackage{mathrsfs}
\usepackage{hyperref}
\usepackage{graphicx}
\usepackage{perpage}
\MakePerPage{footnote}
\usepackage{setspace}
\date{\today}
          
 \usepackage[usenames,dvipsnames]{pstricks}
\usepackage[shell]{pdftricks}
 \begin{psinputs}
 \usepackage{pst-grad} 
 \usepackage{pst-plot} 
 \usepackage{epsfig}
 \end{psinputs}

\title{Detecting Thom faults in stratified mappings}

\author{Saurabh Trivedi \& David Trotman}

\address{Institute of Mathematics of the Polish Academy of Sciences (division in Krakow), Ul. Tomasza 30, 31-027  Krakow, Poland}

\email{saurabh.trivedi@gmail.com}

\address{LATP (UMR 7353), Centre de Math\'ematiques et Informatique, Aix-Marseille Universit\'e, 39 rue Joliot-Curie, 13453 Marseille Cedex 13, France.}

\email{trotman@cmi.univ-mrs.fr}

\begin{document}
\newtheorem{thm}{Theorem}[section]
\newtheorem{lem}[thm]{Lemma}

\theoremstyle{theorem}
\newtheorem{prop}[thm]{Proposition}
\newtheorem{defn}[thm]{Definition}
\newtheorem{cor}[thm]{Corollary}
\newtheorem{example}[thm]{Example}
\newtheorem{xca}[thm]{Exercise}

\newcommand{\bb}{\mathbb}
\newcommand{\al}{\mathcal}
\newcommand{\ak}{\mathfrak}
\newcommand{\fs}{\mathscr}

\theoremstyle{remark}
\newtheorem{rem}[thm]{Remark}
\parskip .12cm

\maketitle
\begin{abstract} 
We state and prove several characterizations of Thom's regularity condition for stratified maps. In particular we extend to stratified maps some characterizations of Whitney (a) regularity, due to the second author.
\end{abstract}

\vspace{-0.8cm}

\section{Introduction}

One of the most important properties of Whitney $(a)$-regularity is that it is necessary and sufficient for the stability of transversality to stratifications. Necessity was proved in \cite{Trotman} by the second author, while Feldman proved sufficiency \cite{Feldman}. It was conjectured in the doctoral dissertation of the second author \cite{Trotman6} that  this characterization can be generalized to the statement that Thom regularity is necessary and sufficient for the stability of transversality to foliated stratifications. 

If the usual Thom transversality theorem were true for transversality to foliations, which is not the case, the conjecture would follow immediately because the method used by the second author would go through. We present a proof of this conjecture using a new method that applies also to the case of  Whitney $(a)$-regularity, moreover we generalize the result of Feldman and Trotman to what we call prestratifications. 

Thom regularity for stratified maps occurs frequently in singularity theory and its applications to dynamical systems, notably in Hilbert's 16th problem about limit cycles, as in the work of Ilyashenko-Kaloshin  \cite{Ilyashenko} and Kaloshin \cite{Kaloshin}, so that equivalent geometric properties are potentially of great interest to specialists. Few previous results exist of this kind, for example see the work of Koike \cite{Koike} and Sch\"urmann \cite{Schurmann}. An important result in equisingularity theory for families of complex hypersurfaces defined by a function $F$, due to L\^e Dung Tr\`ang and K. Saito,  says that the family has constant Milnor number if and only if $F$ is Thom regular \cite{Saito}. Here we present some more geometric properties of Thom regularity.

Section 2 defines the notions of prestratifications, stratifications, Thom regularity for stratified maps, faults and detectors. 

Section 3 shows that  Thom regularity is sufficient for the stability of transversality of foliated prestratifications. This generalizes greatly Proposition 3.6 on page 196 in Feldman \cite{Feldman}.

Section 4 opens with some examples showing that transversality to foliations is not a generic condition and that we cannot detect Thom faults using embeddings. This shows that the method of the second author from \cite{Trotman} in proving the necessity of Whitney $(a)$-regularity for the stability of transversality does not work as a way of proving necessity of Thom regularity for stability of transversality to foliated stratifications. A new method is presented here allowing us to give a short proof of this result. This exploits the fact that a map is transverse to any submanifold if its rank is equal to the dimension of  the target manifold.

In Section 5 after recalling the definition of $(t_f)$-regularity we show that it implies Thom regularity in the case of subanalytic stratifications, where the curve selection lemma holds. That $(a_f)$ implies $(t_f)$ is trivial since  spanning is an open condition and it is used in several places in Kaloshin \cite{Kaloshin}.

In Section 6 we prove the equivalence of Thom regularity and a geometric version of Thom regularity analogous to the geometric version of Whitney $(a)$-regularity that was conjectured  to be equivalent to $(a)$-regularity by Wall \cite{Wall3} and proved to be equivalent by Trotman \cite{Trotman4}, then by Hajto \cite{Hajto} and Perkal \cite{Perkal}. We use a foliated version of the perturbation lemma of Perkal \cite{Perkal}. A similar result is proved by Koike \cite{Koike} but our method is much simpler.

\section{Definitions}

\subsection{Prestratifications and stratifications}

Let $V$ be a closed subset of a $C^1$-manifold $N$. A \emph{prestratification} $\Sigma$ of $V$ is a collection of pairwise disjoint subsets $\{S_{\alpha}\}_{\alpha \in \Lambda}$ of $V$ such that:

1. $\cup_{\alpha \in \Lambda} S_{\alpha} = V$.

2. For every $\alpha \in \Lambda$, $S_{\alpha}$ is an embedded connected submanifold of $N$. We call $S_{\alpha}$'s strata of $\Sigma$. 

3. Every point in $V$ has a neighbourhood in $N$ which intersects only finitely many strata. This is called local finiteness.

By the frontier of a subset $S \subset N$ we mean $\overline S \setminus S$. A prestratification is said to be a \emph{stratification} if it satisfies the \emph{frontier condition}, i.e. the frontier of every stratum is a union of some other strata.

\begin{figure}[hi]
\begin{center}
\includegraphics{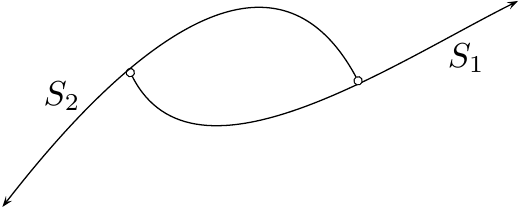}
\caption{Prestratification}\label{nofront}
\end{center}
\end{figure}

In a prestratification it is possible that none of the strata be a closed set, see Figure \ref{nofront}.

\subsection{Thom regularity} Let $N$ and $P$ be $C^1$-manifolds and $f : N \rightarrow P$ be a $C^1$-map of constant rank on a submanifold $S$ of $N$. Then, $f$ induces a foliation on $S$, denoted $\ak F^f_X$. If $x \in S$, we denote by $T_x(\ak F^f_S)$ the tangent space of the leaf of the foliation $\ak F^f_S$ passing through $x$. These notations will be used everywhere in the article.

Let $X$ and $Y$ be submanifolds of $N$ such that $f$ has constant rank on both $X$ and $Y$. The submanifold $X$ is said to be Thom $(a_f)$-regular over $Y$ at a point $y \in Y \cap \overline X$, if 

{\bf $\boldsymbol{(a_f)}$-regularity - }for every sequence $\{x_i\}$ of points in $X$ converging to $y$ such that $T_{x_i}(\ak F^f_X)$ converges to $\tau$, we have $T_y(\ak F^f_Y) \subset \tau$. 

This definition is equivalent to the original definition of Thom regularity as given in Mather \cite{Mather3}. 

Let $\Sigma$ be a prestratification of a closed subset $V$ in $N$ such that $f$ has constant rank on every stratum of $\Sigma$. We will call such a map a \emph{stratified map}\index{Stratified map} though the usual definition of a stratified map is stronger, see Koike \cite{Koike}. Let $X$ and $Y$ be two strata of $\Sigma$. Then the pair $(X,Y)$ is said to be $(a_f)$-regular if $X$ is $(a_f)$-regular over $Y$ at every point in $Y \cap \overline X$ and $Y$ is $(a_f)$-regular over $X$ at every point in $X \cap \overline Y$. Also $\Sigma$ is said to be $(a_f)$-regular if every pair of strata in $\Sigma$ is $(a_f)$-regular.

We show by examples that Thom $(a_f)$-regularity does not imply Whitney $(a)$-regularity and vice-versa; see Mather \cite{Mather3} for the definition of Whitney $(a)$ regularity. However, if $f$ is constant on the strata of a prestratification $
\Sigma$ then $(a_f)$ in this case is equivalent to $(a)$. 

1. Let $N = \bb R^3$, $S_1 = \{z = 0, y >0\}$ and $S_2 = \{y = 0 \}$. Then, $S_1$ is not $(a)$-regular over $S_2$ at any point on $x$-axis. 
 
 Define $f : \bb R^3 \rightarrow \bb R$ by $f(x,y,z) = y + z$. Then the resulting foliated prestratification is $(a_f)$-regular. See Figure \ref{thom}.

2. Let $N = \bb R^3$, $S_1 = \{y>0, z <0, y = z^2\}$ and $S_2=\{y=0\}$. Then, $S_1$ is $(a)$-regular over $S_2$ at every point on the $x$-axis.

 Define $f : \bb R^3 \rightarrow \bb R$ by $f(x,y,z) = y$.  Notice that the fibers of $f$ give a foliation of $S_1$ whose leaves are lines parallel to $x$-axis that lie on $S_1$. The foliation on $S_2$ induced by $f$ is $S_2$ itself. The resulting foliated prestratification is not $(a_f)$-regular. See Figure \ref{thom}.

\begin{figure}[hi]
\begin{center}
\includegraphics{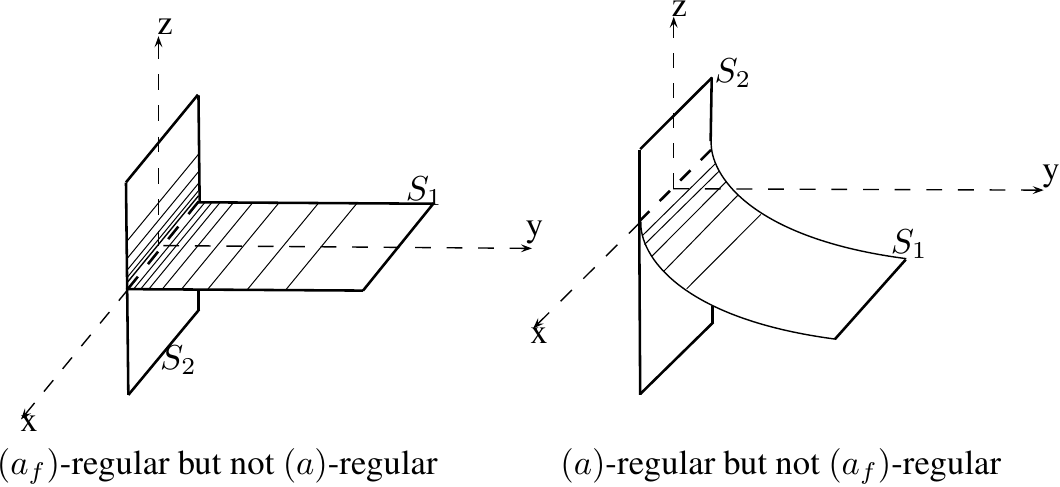}
\caption{Thom regularity}\label{thom}
\end{center}
\end{figure}

\subsection{Faults and Detectors} When some regularity condition $E$ is not satisfied at a point of a prestratification, we call the point an $E$-fault. Many proofs showing that one regularity condition implies another are by contradiction; we suppose that the second condition fails, and then we show that the first condition necessarily fails as well. When we can do this we say we have detected the fault (the point where the second condition fails).

\section{Transversality to Thom regular stratifications is a stable condition} Let $M$, $N$ and $P$ be $C^1$-manifolds. Recall that a $C^1$-map $g: M \rightarrow N$ is said to be transverse to a submanifold $S$ of $N$ at a point $w \in M$, denoted $g \pitchfork_w S$, if either $g(w) \not\in S$ or $g(w) \in S$ and $Dg_w(T_wM) + T_{g(w)}S = T_{g(w)} N$. If $g$ is transverse to $S$ at all point $w\in M$ we write $g \pitchfork S$. 

If $S$ is a stratum of a prestratification $\Sigma$ of a closed subset in $N$ and $f: N \rightarrow P$ is a stratified map, then
\begin{align*}
g \pitchfork \ker d_x(f|_S) \,\,\forall x \in S & \Leftrightarrow  g \pitchfork \ak F^f_S\\
& \Leftrightarrow   g \pitchfork \text{fibres of} \,\,f|_S\\
& \Leftrightarrow  f|_S \circ g : M \rightarrow f(S)\,\, \text{is a submersion}.
\end{align*}

We write $g \pitchfork_K \ak F^f_{\Sigma}$ to say that $f$ is transverse to every leaf of every stratum of $\Sigma$ at points of $K \subset M$. If $K = M$ we simply write $f \pitchfork \ak F^f_{\Sigma}$. Denote by $C^1(M,N)$ the set of all $C^1$-maps between $M$ and $N$. We prove:

\begin{thm}\label{thm21} Let $\Sigma$ be a prestratification of a closed subset $V$ of a $C^1$-manifold $N$, and let $f : N \rightarrow P$ be a stratified map. If $\Sigma$ is an $(a_f)$-regular prestratification, then

i. for every $C^1$-manifold $M$ and any compact set $K \subset M$, $T_K =\{g \in C^1(M,N) : g \pitchfork_K \ak F^f_{\Sigma}\}$ is open in $C^1(M,N)$ with the weak topology. 

ii. for every $C^1$-manifold $M$ and any closed set $K \subset M$, the set $T_K = \{g \in C^1(M,N) : g \pitchfork_K \ak F^f_{\Sigma}\}$ is open in $C^1(M,N)$ with the strong topology.
\end{thm}

To prove Theorem \ref{thm21} we need the following lemma.

\begin{lem}\label{osti}  Let $M$, $N$ and $P$ be $C^1$-manifolds and let $V \subset N$ be a closed set. Let $\Sigma$ be a prestratification of $V$ and $f: N \rightarrow P$ be a stratified map such that $\Sigma$ is $(a_f)$-regular. Let $g : M \rightarrow N$ be a $C^ 1$ map. If $w \in M$ is such that $g \pitchfork_w \ak F^f_\Sigma$  then there exists a coordinate chart $(\phi,\al U)$ of $M$  at $w$ and $(\psi,\al W)$ at $g(w)$ such that for each compact $K \subset \al U$ there is a weak neighbourhood\footnote{Set of all $C^1$-maps $h:M\rightarrow N$ such that $h(K) \subset \al W$, $||g_{\phi,\psi}(x) - h_{\phi,\psi}(x)|| < \epsilon$ and $||Dg_{\phi,\psi}(x) - Dh_{\phi,\psi}(x)|| < \epsilon$ for all $x \in \phi(K)$.} $\al N(g,(\phi,\al U),(\psi,\al W), K,\epsilon)$ of $g$ each of whose members $h$ satisfies $h\pitchfork_K \ak F^f_{\Sigma}$. 
\end{lem}

\begin{proof} Due to local finiteness of $\Sigma$ without loss of generality we can assume that $\Sigma$ has only two strata $X$ and $Y$. We have three cases, 

(1) $g(w) \notin V$,

(2)  $g(w) \in V$ but $g(w) \notin Y \cap \overline{X}$ and $g(w) \notin X \cap \overline{Y}$.

(3) $g(w) \in V$ and $g(w) \in X \cap \overline{Y}$ or $g(w) \in Y \cap \overline X$.

In the first two cases the result follows easily since $V$ is a closed set and the set of surjective linear maps forms an open set. This leaves the only interesting case (3).

So, suppose $g(w) \in V$ and $g(w) \in Y \cup \overline X$. Since $g(w) \in Y$, by the case (2) we can find coordinate charts $(\phi'',U'')$ around $w$ and $(\psi'',W'')$ around $g(w)$ such that for each compact $K \subset U''$ there is a weak subbasic neighbourhood $\al N(g,(\phi'', U''),(\psi'',W''),K,\epsilon'')$ such that every member of this neighbourhood is transverse to every leaf of the foliation $\ak F^f_Y$. 

Now, suppose that, contrary to the conclusion of the lemma, for each neighbourhood $U$ of $x$, there is a compact set $K \subset U$ such that every weak subbasic neighbourhood of $g$ contains a map which is not transverse to one of the leaves of the foliation $\ak F^f_X$ induced on $X$ by $f$.   

Choose $\{U_i\}_{i=1}^{\infty}$ to be a basis for the neighbourhoods of $w$. Then, for each $i$ there is a compact set $K_i \subset U_i$, a point $w_i \in K_i$ and a map $g_i \in \al N(g,(\phi|_{U_i},U_i),(\psi, W), K_i, 1/i)$ such that $g_i \not\pitchfork_{w_i} {\ak F^f_{X_{w_i}}}$, where $(\phi, U)$ and $(\psi, W)$ are fixed charts for $M$ and $N$ at $w$ and $g(w)$ respectively and $\ak F^f_{X_{w_i}}$ is the leaf of $\ak F^f_X$ passing through $g_i(w_i)$. 

Note first that for each $i$, we have $|\psi g_i(w_i) - \psi g(w_i)| < 1/i$
and also that there exist $\{\epsilon_i >0\}$ such that $\epsilon_i \rightarrow 0$ as $i \rightarrow \infty$ and $|\psi g(w_i) - \psi g(w)| < \epsilon_i$. Then by the triangle inequality, $g_i(w_i) \rightarrow g(w)$ as $i \rightarrow \infty$.

Since $g_i \not\pitchfork_{w_i} \ak F^f_{X_{w_i}}$, it follows that
$$ \dim T_{g_i(w_i)} N  > \dim\left (T_{g_i(w_i)}{\ak F^f_{X_{w_i}}}+ D_{w_i}g_i(T_{w_i}M)\right ).$$

Taking limits on both sides and using the properties of sequences of points in Grassmannians we have:
\begin{align*} 
\dim T_{g(w)}N & > \lim_{i\rightarrow  \infty} \dim \left (T_{g_i(w_i)}\ak F^f_{X_{w_i}}+ D_{w_i}g_{w_i}(T_{w_i}M)\right )\\
& =  \dim \lim_{i\rightarrow \infty} \left (T_{g_i(w_i)}\ak F^f_{X_{w_i}}+ D_{w_i}g_{w_i}(T_{w_i}M)\right )\\
& \geq \dim \left (\lim_{i\rightarrow \infty} T_{g_i(w_i)}\ak F^f_{X_{w_i}}+ \lim_{n\rightarrow \infty}  D_{w_i}g_{w_i}(T_{w_i}M)\right )\\
 &\geq \dim \left (\lim_{i\rightarrow \infty} T_{g_i(w_i)}\ak F^f_{X_{w_i}}+  D_wg(T_wM)\right ).
 \end{align*}
 
 But since $X$ is $(a_f)-$regular over $Y$ at $g(w)$ we have
$$\lim_{i\rightarrow \infty} T_{g_i(w_i)}{\ak F^f_{X_{w_i}}}\supset T_{g(w)} {\ak F^f_{Y_{w}}},$$
where $\ak F^f_{Y_w}$ is the leaf of the foliation $\ak F^f_Y$, induced by $f$ on $Y$, passing through $g(w)$. 

Thus it follows that
$$\dim T_{g(w)}N >  \dim\left ( T_{g(w)} {\ak F^f_{Y_w}}+ D_wg(T_wM)\right)$$
which is a contradiction to the fact that $g \pitchfork_w \ak F^f_{\Sigma}$. Thus, there exists a chart $(\phi',U')$ around $w$ and a chart $(\psi',W')$ around $g(w)$ such that for each compact $K \subset U'$ the subbasic neighbourhood of $g$, $\mathcal N(g,(\phi', U'),(\psi', W'),K,\epsilon')$ has the property that all its members are transverse to $ {\ak F^f_X}$ on all of $K$. 

Set $\al U = U'\cap U''$ and $\al W = W'\cap W''$. It is easy to see that for a suitable $\epsilon$ and any compact $K \subset \al U$, the subbasic neighbourhood $\mathcal N(g,(\phi, \al U),(\psi, \al W),K,\epsilon)$ satisfies,
\begin{align*}
\mathcal N(g,(\phi, \al U),(\psi, \al W),K,\epsilon) &\subset \mathcal N(g,(\phi, U'),(\psi, W'),K,\epsilon') \,\bigcap\\
&\qquad \mathcal N(g,(\phi'', U''),(\psi'', W''),K,\epsilon'')
\end{align*}
and all its members are transverse to $\ak F^f_X$ and $\ak F^f_Y$ on $K$.
\end{proof}

\begin{proof}[Proof of Theorem \ref{thm21}] The two parts will be treated separately.

{\bf i.} To prove that $T_K$ is open, we show that there exists a weak open neighbourhood of every map in $T_K$ contained in $T_K$. Take a map $g \in T_K$, since $g$ is transverse to $\ak F^f_\Sigma$ at each $w \in K$ and $\Sigma$ is $a_f$-regular, by Lemma \ref{osti}, for each $w \in K$ there exists a chart $U_w$ with the property that for each compact set $K_w \subset U_w$ there is a neighbourhood $\al N(g,(\phi_w,U_w),(\psi_w,V_w),K_w,\epsilon_w)$ such that each member of this neighbourhood is transverse to $\Sigma$ on all of $K_w$. Since $K$ is compact, we can choose a finite subcollection $\{U_{w_1},\ldots,U_{w_r}\}$ of the coordinate neighbourhoods $\{U_w\}_{w\in K}$, such that $K \subset \cup_{i=1}^r K_{w_i}$. But then the intersection 
$$\cap_{i=1}^r \al N(g,(\phi_{w_i},U_{w_i}),(\psi_{w_i},V_{w_i}),K_{w_i},\epsilon)$$
      ($\epsilon = \min \{\epsilon_{w_i}\}$) is a weak open neighbourhood of $g$ and is contained in $T_K$, as required.
      
{\bf ii.} First notice that every open covering of a  closed set of a smooth manifold has a locally finite open refinement. Now as in (i), by Lemma \ref{osti}, for each $w \in K$ there is a chart $U_w$ for $M$ which contains $w$ and has the property: for each compact set $K_w \subset U_w$ there is a neighbourhood $\al N(g,(\phi_w, U_w),(\psi_w,V_w), K_w,\epsilon_w)$ all of whose members are transverse to $\ak F^f_{\Sigma}$. Now choose a locally finite subcollection of the charts $\{U_w\}_{w\in K}$ which covers $K$. By the definition of the strong topology, the intersection of the weak subbasic neighbourhoods for this finite subcollection of charts gives a strong open neighbourhood all of whose members are transverse to $\ak F^f_{\Sigma}$ on $K$.
\end{proof}

\section{Stability of transversality implies Thom regularity} 

First we show that transversality of maps to foliations is not a generic condition. Consider the following examples. 

1. Let $M = S^1$ and $N = \bb R^2$ foliated by lines parallel to $x$-axis. Then the embedding of $M$ into $N$ is non transverse to this foliation and it cannot be made transverse by small perturbations. See Figure \ref{generic}.

2. Let $M = \bb R$ and $N = \bb R^2$ foliated by lines parallel to $x$-axis. Let $f : M \rightarrow N$ be given by $f(x) = (x,x^3-x)$. Then $f$ is non transverse to this foliation and it cannot be made transverse by small perturbations. See Figure \ref{generic}.

\begin{figure}[hi]
\begin{center}
\includegraphics{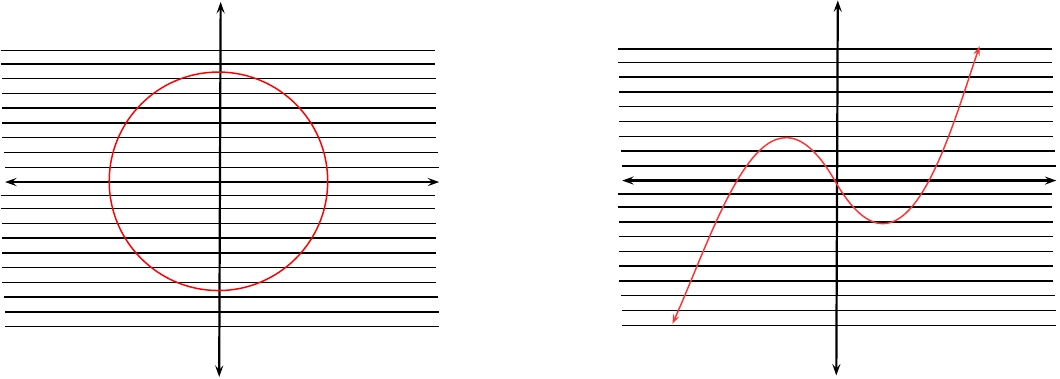}
\caption{Non-genericity}\label{generic}
\end{center}
\end{figure}

{\bf 2.} Let $M=S^2$, $N = \bb R^2$ and $S$ be the unit circle in $N$ considered as foliated by its points. Project $M$ onto $N$ in a way that the image of $M$ under this projection $f$ is a disc $D$ which does not entirely cover $S$. The rank of this projection $f$ at the points whose image intersect $S$ and lie in the interior of $D$ is $2$ and thus $f$ is transverse to the foliated circle at these points. But, we cannot find maps close to $f$ which are transverse on every point of $M$ because any sufficiently small perturbation of $f$ will not cover $S$ entirely and whenever the boundary of the image intersects the circle it will not be transverse to the foliated circle.

Secondly, we show that we cannot detect $(a_f)$-faults by embeddings. 

{\bf 3.} Consider the blow-up of $\bb R^2$ at $0$, given by $\beta : E \rightarrow \bb R^2$ where $E$ is the canonical line bundle over the real projective line $\bb R P^1$. Recall that $E$ is topologically a M\"obius band embedded into $\bb R^3$. We take the stratification of $E$ with two strata, $X = \beta^{-1}(0)$ and $Y = E\backslash X$. The blow-up map $\beta$ induces a foliation on $X$ and $Y$. The only leaf of this foliation of $X$ is $X$ itself and the leaves of the foliation of $Y$ are points. It is easy to see that this stratification of $E$ is not $a_{\beta}$-regular. 

Notice that we cannot detect the $a_{\beta}$-faults in this stratifications by perturbing embeddings of rank 2 or even rank 3. For, no embedding of rank 2 can intersect this foliated stratification transversely since by the definition of transversality we need to have a rank 3 map to intersect the foliation of $Y$. Thus the set of maps transverse to this stratification from a manifold of dimension $2$ is empty and so open and yet our stratification is not $a_{\beta}$ regular. 

\begin{figure}
\begin{center}
\includegraphics[scale=.75]{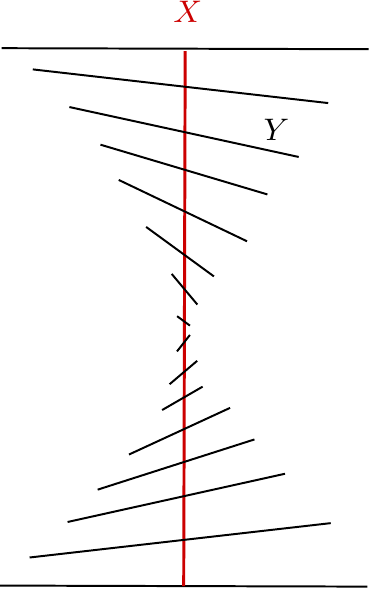}
\caption{Stratification of a Blow-up}
\end{center}
\end{figure}

Similarly any embedding of rank 3 is always transverse to this foliated stratification by the definition of transversality and thus we cannot find a sequence of embeddings of rank 3 not transverse to one foliated stratum but whose limit is transverse to the other foliated stratum.

This shows that the method used to prove that Whitney $(a)$-regularity is necessary and sufficient for the stability of transversality by the second author in \cite{Trotman} does not work to obtain a generalization for the Thom regularity. We will show however using a new method  that Thom regularity is necessary and sufficient for the stability of transversality to foliated stratifications. 

We will prove the following result:

\begin{thm}\label{thm31} Let $N$ and $P$ be $C^1$-manifolds. Let $f: N \rightarrow P$ be a stratified map for a prestratification $\Sigma$ of a closed subset $V$ of $N$. Then the following are equivalent:

(1) $\Sigma$ is $(a_f)$-regular,

(2) for any $C^1$-manifold $M$, the set $\{g \in C^1(M,N) : g \pitchfork \ak F^f_{\Sigma}\}$ is open in the strong topology,

(3) the set $\{g \in C^1(N,N) : g \pitchfork \ak F^f_{\Sigma}\}$ is open in the strong topology.
\end{thm}

To prove Theorem \ref{thm31} we need the following lemma:

\begin{lem} \label{lem32}Let $N$ be a smooth manifold of dimension $n \geq 2$ and let $x \in N$. Let $r$ be a positive integer strictly less than $n$. Then, there exists a bijective smooth map $g : N \rightarrow N$ such that 

i. the rank of $g$ at $x$ is $r$,

ii. the rank of $g$ at all points except $x$ is $n$.
\end{lem}

\begin{proof} Take a chart $(U,\phi)$ (coordinate disk) around $x$. We will construct a map $g_U : U \rightarrow U$ that has the above properties on $U$, $g_U(x) = x$ and $g_U$ is identity outside a relatively compact subset of $U$. The lemma will follow by extending this map outside $U$ by the identity map.

Let $\gamma : \bb R \rightarrow \bb R$ be a smooth map with following properties. 

i. $\gamma(a) =0$ if $a \leq0$,

ii. $\gamma(a) = 1$ if $a\geq1$,

iii. $\gamma'(a) > 0$ for $0 <a<1$.

Define a map $h : \bb R^n \rightarrow \bb R^n$ by 
$$h(a_1,\ldots,a_n)=(a_1,\ldots,a_r,a_{r+1}\gamma(\|a\|^2),\ldots,a_n\gamma(\|a\|^2)).$$
where $a=(a_1,\ldots, a_n) \in \bb R^n$ and $\|a\|^2= a_1^2 + \cdots + a_n^2$.

Notice that the map $h$ is smooth and has the following properties:

i. it has rank $r$ at $0$,

ii. it has rank $n$ at any point $a\neq 0$,

iii. it is the identity outside the ball of radius 1,

iv. it is bijective.

Now define $g_U : U \rightarrow U$ by $g_U = \phi^{-1} \circ h \circ \phi$. Then, $h_U$ has the required properties. 
\end{proof}

\begin{proof}[Proof of Theorem \ref{thm31}] The implications ((1) $\Rightarrow$ (2)) and ((2) $\Rightarrow$ (3))  follow from Theorem \ref{thm21}. The only implication to be proved is ((3) $\Rightarrow$ (1)).

Suppose $\Sigma$ is not an $(a_f)$-regular prestratification. Then, there exists a sequence $\{x_i\}$ in a stratum $X$ converging to a point $y$ in a stratum $Y$ such that $T(\ak F^X_{x_i}X)$ converges to $\tau$ in the Grassmannian but $\tau$ does not contain $T_{y} (\ak F^f_Y)$ as a subspace. Let $v \in T_y(\ak F^f_Y)$ such that $v \not\in \tau$. Then, there exists a subspace $H$  of $T_y N$ of dimension $n - \dim T_y({\ak F^f_Y})$, ($\dim N = n$), not containing $v$ such that
\begin{align}
H \oplus T_y(\ak F^f_Y) &= T_yN \label{eq1}\\
H + \tau &\neq T_yN \label{eq2}
\end{align}

Moreover, there exists a sequence $\{H_i\}$ of subspaces of $T_{x_i}N$ of dimension $n - \dim T_y(\ak F^f_Y)$ such that for large enough $i$,
\begin{equation} 
H_i + T_{x_i}(\ak F^f_X) \neq T_{x_i}N. \label{eq3}
\end{equation}

By Lemma \ref{lem32} there exists a map $g : N \rightarrow N$ (after a suitable change of coordianates) with the following properties:

1. $g(y) = y$,

2. the rank of $g$ at $y$ is $n - \dim T_y(\ak F^f_Y)$,

3. the rank of $g$ at any point other than $y$ is $n$, and

4. $D_yg(T_yN)=H$.

We have $g \pitchfork \ak F^f_{\Sigma}$ since the rank of $g$ is $n$ at all points except $y$ and by (\ref{eq1}) $g$ is transverse to $\ak F^f_Y$ at $y$. 

Now it is easy to find a sequence of maps $\{g_i : N \rightarrow N\}$ converging to $g$ in the strong topology such that for large $i$, $g_i(y) = x_i$, and
\begin{equation}
D_yg_i(T_yN) = H_i. \label{eq4}
\end{equation}

By (\ref{eq3}) and (\ref{eq4}), for large $i$, $g_i \not\pitchfork \ak F^f_X$, which is a contradiction to the hypothesis that $\{g \in C^1(N,N) : g \pitchfork \ak F^f_{\Sigma}\}$ is open in the strong topology.
\end{proof}

\section{Thom regularity and $(t_f)$-regularity}

Let $N$ and $P$ be $C^1$-manifolds and $X$ and $Y$ be submanifolds of $N$. Let $f : N \rightarrow P$ be a $C^1$-map that has constant rank of $X$ and $Y$. Then, $X$ is said to be $(t_f)$-regular over $Y$ at $y \in Y \cap \overline X$ if 

{\bf $\boldsymbol{(t_f)}$-regularity} - Given a $C^1$ submanifold $S$ of $N$ transverse to the leaf of $\ak F^f_Y$ passing through $y$, there is a neighbourhood $U$ of $y$ in $N$ such that $S$ is transverse to $\ak F^f_X$ in $U$. 

Since spanning is an open condition, it follows at once that $(a_f)$-regularity implies $(t_f)$-regularity. Kaloshin \cite{Kaloshin} uses and gives a proof of this trivial fact in his article on Hilbert's sixteenth problem, see page 463, page 492 and proposition 2 in page 495 in \cite{Kaloshin}. 

We show that $(a_f)$-faults can be detected by $(t_f)$-regularity in the subanalytic case. Since the arguments are local we work with $\bb R^n$.

\begin{thm}\label{thm121} Let $X$, $Y$ be $C^1$-submanifolds of $\bb R^n$ with $0 \in Y \cap \overline{X}$, and let $Y$ be a subanalytic set. Let $f: \bb R^n \rightarrow \bb R^p$ be a subanalytic map (i.e. the graph of $f$ is subanalytic in $\bb R^n \times \bb R^p$), such that $f|_X$ and $f|_Y$ are of constant rank. Then $X$ is $(a_f)-$regular over $Y$ at $0$ if and only if for every semianalytic  $C^1$ submanifold $S$ transverse to $\ak F^f_Y$ at $0$, there is some neighbourhood of $0$ in which $S$ is transverse to $\ak F^f_X$.
\end{thm} 

\begin{proof}  Condition $(a_f)$ implies $(t_f)$ trivially. We show that $(t_f)$ implies $(a_f)$ with the hypothesis of the theorem.

Let $0 \in Y$ be an $(a_f)$-fault.  Then, there exists a sequence $\{x_i\}$ in $X$ converging to $0$ such that the limit of the tangent spaces $T_{x_i}(\ak F^f_X)$ converg es to $\tau$, but $\tau$ does not contain $T_0 (\ak F^f_Y)$. 

We will construct a semianalytic $C^1$ submanifold $S$ transverse to the leaf of $\ak F_Y$ passing through $0$ but  not transverse to the leaves of $\ak F^f_X$ on any neighbourhood around $0$.

Since $T_0 (\ak F^f_Y) \not\subset \tau$ there exists a unit vector $v \in T_0 (\ak F^f_Y)$ such that $v \not\in \tau$. This implies that there exist $\epsilon > 0$ and a positive number $n$ such that for all $i>n$, 
$$d(v, T_{x_i}(\ak F^f_X)) > \epsilon.$$
where $d(v, T_{x_i}(\ak F^f_X))$ denotes the distance between $v$ and $T_{x_i}(\ak F^f_X)$.

Let $m$ be the dimension of leaves of $\ak F^f_X$ and let $$V_1 = \bb R^n \times \{P \in G^n_m(\bb R) : d(v,P) > \epsilon\}$$
and
$$V_2 = \{(x, T_x(\ak F^f_X)): x \in X\} \subset \bb R^n \times G^n_m(\bb R),$$
where $G^n_m(\bb R)$ denotes the Grassmann bundle of $m$-dimensional subspaces of $\bb R^n$. 

The set $V_1$ is semialgebraic and we show that $V_2$ is subanalytic. In fact we just need to show that $\{(x, T_xX) : x \in X\}$ is subanalytic, which is precisely Lemma 1.6 in Verdier \cite{Verdier}. For, $T_x(\ak F^f_X) = \ker d_xf \cap T_x X$, and $\ker d_y f$ is a fixed subspace of $\bb R^n$ if we suppose that $f$ is a linear projection  (as we can since $f$ is the composition of an embedding onto its graph followed by a linear projection (cf. page 30 of Teissier \cite{Teissier})). Thus, $V_2$ is a subanalytic set. Semialgebraic sets are subanalytic, and the finite intersection of subanalytic sets is subanalytic (by Hironaka \cite{Hironaka}). Hence, $V_1 \cap V_2$ is subanalytic.

Notice that $(0,\tau) \in \overline{V_1 \cap V_2}$, thus by the curve selection lemma (see Proposition 3.9 in Hironaka \cite{Hironaka}), there is an analytic arc
$$\alpha : [0,1] \rightarrow \bb R^n \times G^n_m$$
given by $\alpha(t) = (\alpha_1(t),T_{\alpha_1(t)} (\ak F^f_X))$, such that $\alpha(0) = (0,\tau)$ and $\alpha(t) \in V_1 \cap V_2$ if $t > 0$.

Denote by $N_t \in G^n_{n-1}(\bb R)$ the orthogonal complement of the tangent space to the manifold-with-boundary $\alpha_1[0,1]$ and let $v_t$ be the orthogonal projection of $v$ into $N_t$. Denote by $\langle v \rangle$ the subspace spanned by $v$.

Let $\sigma: [0,1] \rightarrow G^n_{n-2}(\bb R)$ be the analytic curve defined by
$$\sigma(t) = P_t \oplus (P_t \oplus \langle v_t\rangle)^{\perp}$$
where $P_t = N_t \cap T_{\alpha_1(t)} (\ak F^f_X)$ and $()^{\perp}$ is the orthogonal complement in $N_t$.

Notice that, $v_t \not\in P_t$ and moreover $\sigma(t) \oplus \langle v \rangle = N_t$. 

Then the union of $\{\sigma(t)\}$ for $t\in [0,1]$, considered as embedded $(n-2)$-planes in $\bb R^n$ passing through points $\alpha_1(t)$ defines a semianalytic manifold-with-boundary $S'$ of dimension $(n-1)$. Reflection in $N_0$ extends $S'$ to a $C^1$-manifold $S$ which is a semianalytic subset of $\bb  R^n$ and which is transverse to $\ak F^f_ Y$ at $0$. 

Finally we show that $S$ is not transverse to $\ak F^f_X$ on any neighbourhood of $0$. Let $U$ be a neighbourhood around $0$. There exists some $t_0 \in (0,1]$ such that $U \cap \alpha_1(0,1] \supset \alpha_1(0,t_0]$. But $S'$ (and hence $S$) is not transverse to $\ak F^f_X$ at any point of $\alpha_1(0,1]$. For, if $A_t$ denotes the tangent space to the curve $\alpha_1(0,1]$ at $\alpha_1(t)$,
$$T_{\alpha_1(t)} \ak F^f_X = P_t \oplus A_t \subset \sigma(t) \oplus A_t = T_{\alpha_1(t)} S.$$ 
Theorem \ref{thm121} follows.
\end{proof}

\section{Geometric versions of Thom regularity}

Let $N$ and $P$ be $C^1$-manifolds and $X$ and $Y$ be submanifolds of $N$. Let $f : N \rightarrow P$ be a $C^1$-map that has constant rank of $X$ and $Y$. Then, $X$ is said to be $(a_f^s)$-regular over $Y$ at $y \in Y \cap \overline X$ if 

{\bf $\boldsymbol{(a_f^s)}$-regularity}  - Given a $C^1$ local retraction $\pi$ defined near $y$ onto the leaf of $\ak F^f_Y$ passing through $y$, there is a neighbourhood $U$ of $y$ in $\bb R^n$ such that $\pi|_{X \cap U}$ is a submersion on every leaf of $\ak F_X^f \cap U$. 

This is equivalent to saying that $X$ is $(\ak F^1_f)$-regular over $Y$ at $y \in Y \cap \overline X$, where we define

{\bf $\boldsymbol{(\ak F^k_f)}$-regularity}  - Given a $C^k$ foliation $\ak G$ of $N$ transverse to $\ak F^f_Y$ at $0$, there is a neighbourhood of $0$ in which $\ak G$ is transverse to $\ak F^f_X$.

We show that (also proved in Koike \cite{Koike} but our proof is much simpler):

 \begin{thm} \label{thm122} Let $f: N \rightarrow P$ be a $C^1$ map, between $C^1$ manifolds $N$ and $P$, $X$ and $Y$ be $C^1$ submanifolds of $N$ such that $f|_X$ and $f|_Y$ have constant rank, and let $0 \in Y \cap \overline{X}$. Then the following conditions are equivalent.
 
 i. $X$ is $(a_f)$-regular over $Y$ at $0$,
 
 ii. $X$ is $(a_f^s)$-regular over $Y$ at $0$.
 \end{thm}
 
 We state the following lemmas that can be obtained by slight modifications to Perkal's perturbation theorems 1.2 and 1.4 in \cite{Perkal}. 

\begin{lem}\label{lem123} Let $X$ be a submanifold of $\bb R^n$, let $x_0 \in \overline{X}$ and let $\{x_i\}$ be a sequence in $X$ converging to $x_0$. For each sequence $\{L_i\}_{i=0}^{\infty}$ of linear bijections from $\bb R^n$ to $\bb R^n$ converging to the identity, there is a $C^1$ chart $(V,\psi)$  around $x_0$ such that $\psi(x_i) = x_i$ and $D{\psi(x_i)}=L_i$ for large $i$.
\end{lem}

\begin{lem}\label{lem124} Let $X$ be a manifold of $\bb R^n$, let $x_0 \in \overline{X}$ and let $\{x_i\}$ be a sequence in $X$ converging to $x_0$. Let $\ak F$ be a foliation of $X$. If $\{T_{x_i} (\ak F)\}$ and a sequence  $\{\tau_i\}$ of linear subspaces of $\bb R^n$ converge to a common limit $\tau$ in the Grassmann bundle then there is a $C^1$ chart $\psi$ around $x_0$ such that for large $i$, $\psi(x_i)=x_i$ and $T_{x_i}(\ak \psi(\ak F)) = \tau_i$.
\end{lem}

\begin{proof}[Proof of Theorem \ref{thm122}] That $(a_f)$ implies $(a^s_f)$ follows very easily following the proofs that $(a)$ implies $(a_s)$ in Wall \cite{Wall3} or page 9 in Thom \cite{Thom2}.

Let $s$ be the dimension of leaves of $\ak F^f_Y$. Since the argument is local, we can assume without loss of generality that $T_0{(\ak F^f_Y)} =\bb R^s \times \{0\}$.

Suppose $X$ is not $(a_f)$-regular over $Y$ at $0$. There exists an $f$-good sequence $\{x_n\}$ in $X$ such that its Grassmann limit does not contain $T_0(\ak F^f_Y)$. Let $(V_1,V_2)$ be the basis for $\tau$ where $V_1$ is a basis of $\tau \cap T_0(\ak F^f_Y)$ and $V_2$ is a basis for the orthogonal complement of $\tau \cap  T_0(\ak F^f_Y)$ in $\tau$. 

Extend the set of vectors consisting of vectors in $V_2$, and a basis for $Y$ to a basis for $\bb R^n$ and let $A$ be the linear map from $\bb R^n$ to $\bb R^n$ that maps this basis to itself.

Notice that $\pi_s(A(V_1)) = V_1$ while $\pi_s(A(V_2)) = 0$ and $\pi_s(A(\tau))=\tau \cap Y$, where $\pi_s : \bb R^n \rightarrow \bb R^s\times \{0\}$ is the natural projection. This gives a linear bijection $A$ such that $\pi_s|_{A(\tau)}$ is not submersive. 

Thus we can assume, after making this linear change of chart, that $\pi|_{\tau}$ is not submersive. By Lemma \ref{lem124} there is a $C^1$ chart $\psi$ such that, for large $i$, $T_{x_i}(\psi(\ak F^f_X)) = \tau$. But
$$D_{x_i}(\pi_s|_{\psi(\ak F^f_X)}) = (D_{x_i} \pi_s)(T_{x_i}(\ak F^f_X)) = \pi_s|_{\tau}$$ so that $\pi_{\psi} = \psi^{-1} \circ \pi_s \circ \psi$ fails to be a submersion on $V \cap X$ for some neighbourhood $V$ of $0$. Hence $(a_f)$-regularity fails.
\end{proof} 

Another proof of the above theorem can be given using the condition  $(\ak F^1_f)$, which is equivalent to $(a_f)$, and Trotman's idea of ripples \cite{Trotman4}. We cannot replace $C^1$ by $C^2$ in the statement of Theorem \ref{thm122}, see Perkal \cite{Perkal} or Kambouchner and Trotman \cite{Trotman5}.

\bibliographystyle{amsplain}
\bibliography{mainbibliography}

\end{document}